\documentclass[11pt,a4paper]{amsart}%
\pdfoutput=1
\usepackage{hyperref}
\usepackage{amsfonts}
\usepackage{amsmath}
\usepackage{amssymb}
\usepackage{amsthm}
\usepackage{graphicx}
\usepackage{algorithm}
\usepackage{a4wide}
\usepackage{algorithmic}%
\usepackage{float}

\usepackage{color} 

\theoremstyle{plain}
\newtheorem{theorem}{Theorem}[section]

\newtheorem{proposition}[theorem]{Proposition}

\theoremstyle{definition}

\newtheorem{example}[theorem]{Example}
\theoremstyle{remark}

\numberwithin{equation}{section}
\theoremstyle{plain}

\newcommand{\erz}[1]{\langle#1\rangle}
\newcommand{\maxk}[1]{\left\{#1\right\}}

\DeclareMathOperator{\reg}{reg}
\DeclareMathOperator{\codim}{codim}
\DeclareMathOperator{\hilb}{Hilb}

\begin{document}
\title[Decomposition of semigroup algebras]{Decomposition of semigroup algebras}
\author{Janko B\"{o}hm}
\address{Department of Mathematics, University of Kaiserslautern, Erwin-Schr\"odinger-Str., 67663 Kaiserslautern, Germany}
\email{boehm@mathematik.uni-kl.de}
\author{David Eisenbud}
\address{Department of Mathematics, University of California, Berkeley, CA 94720, USA}
\email{de@msri.org}
\author{Max J. Nitsche}
\address{Max\mbox{\;}Planck\mbox{\;}Institute\mbox{\;}for\mbox{\;}Mathematics\mbox{\;}in\mbox{\;}the\mbox{\;}Sciences,\mbox{\;}Inselstrasse\mbox{\;}22,\mbox{\;}04103\mbox{\;}Leipzig,\mbox{\;}Germany}
\email{nitsche@mis.mpg.de}
\thanks{}
\date{\today}

\begin{abstract}
Let $A\subseteq B$ be cancellative abelian semigroups, and let $R$ be an integral domain. We show that the semigroup ring $R[B]$ can be decomposed, 
as an $R[A]$-module, into a direct sum of $R[A]$-submodules of the quotient ring of $R[A]$. In the case of a finite extension of positive affine semigroup rings we obtain an algorithm computing the decomposition. When $R[A]$ is a polynomial ring over a field we explain how to compute many ring-theoretic properties of $R[B]$ 
in terms of this decomposition. In particular we obtain a fast algorithm to compute the Castelnuovo-Mumford regularity of homogeneous semigroup rings. As an application we confirm the Eisenbud-Goto conjecture in a range of new cases. Our algorithms are implemented in the \textsc{Macaulay2} package \textsc{MonomialAlgebras}.

\end{abstract}
\maketitle

\section{Introduction}

Let $A\subseteq B$ be cancellative abelian semigroups, and let $R$ be an integral domain. Denote by $G(B)$ the group generated by $B$, and by $R[B]$ the semigroup ring associated to $B$, that is, the free $R$-module with basis formed by the symbols $t^a$ for $a\in B$, and multiplication given by the $R$-bilinear extension of $t^a\cdot t^b=t^{a+b}$. 
Extending a result of
Hoa and St\"uckrad in \cite{HSCM}, we show that the semigroup ring $R[B]$ can be decomposed, 
as an $R[A]$-module, into a direct sum of $R[A]$-submodules of $R[G(A)]$ indexed by the elements of the factor group $G(B)/G(A)$.

By a \emph{positive affine semigroup} we mean a finitely generated subsemigroup $B\subseteq \mathbb N^m$, for some $m$.
If $A\subseteq B\subseteq  \mathbb N^m$ are positive affine semigroups,
$K$ is a field, and the positive rational cones $C(A)\subseteq C(B)$ spanned by $A$ and $B$ are equal, then
$K[B]$ is a finitely generated $K[A]$-module and we can make the decomposition above effective. In this case the number of submodules $I_g$ in the decomposition is finite, and we can choose them to be ideals of $K[A]$. We give an algorithm for computing the decomposition, implemented in our \textsc{Macaulay2} \cite{M2} package \textsc{MonomialAlgebras} \cite{BEN}. \\

By a \emph{simplicial semigroup}, we mean a positive affine semigroup $B$ such that $C(B)$ is a simplicial  cone.
If $B$ is simplicial and $A$ is a subsemigroup generated by elements on the extremal rays of $B$,
many ring-theoretic properties of $K[B]$ such as 
being Gorenstein, Cohen-Macaulay, Buchsbaum, normal, or seminormal, 
can be characterized in terms of the decomposition, see Proposition~\ref{char}. Using this we can provide functions to test those properties efficiently.\\

Recall that any positive affine semigroup $B$ has a unique minimal generating set
called its \emph{Hilbert basis} $\hilb(B)$. By a \emph{homogeneous semigroup} we mean
a positive affine semigroup that admits an $\mathbb N$-grading where all the elements of $\hilb(B)$ have degree 1.

One motivation for developing the decomposition algorithm was to have a more efficient
algorithm to compute the Castelnuovo-Mumford regularity (see Section~\ref{sec simplicial homogeneous case} for the definition) 
of a homogeneous semigroup ring $K[B]$. This invariant
 is often computed from a minimal graded free resolution of $K[B]$ as a module over a polynomial ring in $n$ variables,
  where $n$ is the cardinality of $\hilb(B)$. The free resolution could have length $n-1$, and if $n$ is large (say $n\geq15$) this computation becomes very slow. But in fact the Castelnuovo-Mumford regularity of $K[B]$ can be computed from 
a minimal graded free resolution of  $K[B]$ as a  module over any polynomial ring, so long as $K[B]$ is finitely generated.
For example, if $A$ is the subsemigroup generated by
elements of $\hilb(B)$ that lie on the extremal rays of $B$, 
and $K[B] \cong \oplus_{g}I_{g}$ is a decomposition  as graded $K[A]$-modules, then the regularity of $K[B]$ is the maximum of
the regularities of the $I_{g}$ as $K[A]$-modules (Proposition~\ref{regcomp}). Since the minimal graded free resolution of $I_g$
has length at most the cardinality of $\hilb(A)$ (equal to the dimension of $K[B]$ in the simplicial case),
and the decomposition can be obtained very efficiently, this method of computing the
regularity is typically much faster. See Section~\ref{sec simplicial homogeneous case} for timings. \\

The Eisenbud-Goto conjecture gives a bound on the Castelnuovo-Mumford regularity
\cite{EG}. It is known to hold in relatively few cases. The efficiency
of our algorithm allows us to test many new cases of the  conjecture (Proposition~\ref{testegbis5}).

\section{Decomposition}

If $X\subseteq G(B)$ we write $t^X:=\{t^x\mid x\in X\}$.

\begin{theorem}\label{decomposition} 

Let $A\subseteq B$ be cancellative abelian semigroups, and let $R$ be an integral domain. The $R[A]$-module $R[B]$ is isomorphic to the direct sum of submodules $I_{g}\subseteq R[G(A)]$ indexed by elements $g\in G:= G(B)/G(A)$.

\end{theorem}
\begin{proof} We think of an element $g\in G$ as a subset of $G(B)$.
For $g\in G$ let
\[
\Gamma_{g}':=\{b\in B\mid b\in g\}.
\]
By construction, we have
\[
R[B]=\bigoplus\nolimits_{g\in G} R\cdot t^{\Gamma_{g}'}.
\]
For each $g\in G$, choose a representative $h_{g}\in g\subseteq G(B)$. The module $R\cdot t^{\Gamma_{g}'}$ is an $R[A]$-submodule of $R[B]$ and, as such, it is isomorphic to
\[
I_{g}:=R\cdot\{t^{b-h_{g}}\mid b\in\Gamma_{g}'\}\subseteq R[G(A)]\text{.}
\]
\end{proof}

With notation as in the proof, we have 
$$
R[B] \cong_{R[A]} \bigoplus\nolimits_{g\in G}I_{g}\cdot t^{h_{g}}.
$$
This decomposition, together with the ring structure of $R[A]$ and the group structure of $G$
actually determines the ring structure of $R[B]$: if $x\in I_{g_{1}}$ and
$y\in I_{g_{2}}$ and $xy=z$ as elements of $R[G(A)]$ then as elements in the
decomposition of $R[B]$%
\[
x\cdot_{R[B]}y=\frac{t^{h_{g_{1}}}t^{h_{g_{2}}}}{t^{h_{g_{1}+g_{2}}}}z\in I_{g_{1}+g_{2}}.
\]
\noindent{\bf Henceforward we assume that  $A\subseteq B\subseteq \mathbb N^m$ are positive affine semigroups, and
we work with monomial algebras over a field $K$. }

The set 
$
B_A=\{x\in B\mid x \notin B+(A\setminus\{0\})\}
$
is the unique minimal subset of $B$ such that $t^{B_A}$ generates $K[B]$ as a $K[A]$-module. We define $\Gamma_g:=\{b\in B_A \mid b\in g\}$. Then $\Gamma_g + A = \Gamma_{g}'$.\\

We can compute the decomposition of Theorem~\ref{decomposition} if  
$K[B]$ is a finitely generated $K[A]$-module, or equivalently 
$B_{A}$ is a finite set. 
This finiteness (for positive affine semigroups $A\subseteq B$) 
is equivalent to the property $C(A)=C(B)$, where $C(X)$ denotes the positive rational cone spanned by $X$ in $\mathbb Q^m$. 
(Proof: if $C(A)\subsetneqq C(B)$ we can choose an element $x\in B$ on a ray of $C(B)$ not in $C(A)$, so $nx\in B_A$ for all $n\in\mathbb N^+$. Thus, $B_A$ is not finite. Conversely, if $C(A)=C(B)$, then for all $b\in B$ there exists  $n_{b}\in\mathbb N^+$ such that $n_bb\in A$. To generate
$K[B]$ as a $K[A]$-module, it suffices to take
all possible sums of the multiples $mb$ such that $m<n_{b}$ for all $b$ in a (finite) generating set for
the semigroup $B$.) Note that if $B_A$ is finite, then $G(B)/G(A)$ is also finite.\\ 

From these observations we obtain Algorithm~\ref{algo1} computing the set $B_A$ and the decomposition of $K[B]$.

\begin{algorithm}[H]           
\caption{Decompose monomial algebra}  
\begin{algorithmic}[1]\label{algo1}

\REQUIRE  A homogeneous ring homomorphism
\[
\psi:K[y_{1},\ldots,y_{d}] \rightarrow K[x_{1},\ldots,x_{n}]
\]
of $\mathbb N^m$-graded polynomial rings over a field $K$ with $\deg y_{i}=e_{i}$ and $\deg x_{j}=b_{j}$ such that $\psi(y_i)$ is a monomial for all $i$ and the gradings specify positive affine semigroups $A=\langle e_{1},\ldots,e_{d}\rangle\subseteq B=\left\langle b_{1},\ldots,b_{n}\right\rangle\subseteq \mathbb N^m$ with $C(A)=C(B)$.

\ENSURE An ideal $I_{g}\subseteq K[A]$ and a shift  $h_g\in G(B)$ for each
$g\in G:= G(B)/G(A)$ with
\[
K[B] \cong\bigoplus\nolimits_{g\in G} I_{g} (-h_{g})
\]
as $\mathbb Z^m$-graded $K[A]$-modules (with $\deg t^b = b$). 

\STATE Compute the set $B_{A} = \{b\in B \mid b\notin B+(A\setminus\{0\})\}$,
and let $\{v_{1},\dots, v_{r}\}$ be the monomials in $K[B]$ corresponding to elements of $B_{A}$. For 
example, this can be done by computing the toric ideal $I_B:=\ker \varphi$ associated to $B$, where 
\[
\varphi: K[x_{1},\ldots,x_{n}] \rightarrow K[B], \quad x_i\mapsto t^{b_i},
\]
and then computing a monomial $K$-basis $v_{1},\ldots,v_{r}$ of
\[
K[x_{1},\ldots,x_{n}]/(I_B+\psi(\langle y_{1},\ldots,y_{d}\rangle)).
\]

\STATE Partition the elements $v_{i}$ by their class modulo $G(A)$,
forming the decomposition
\[
B_A=
{\dot{\bigcup}}_{g\in G}\Gamma_g.
\]

\STATE For each $g\in G$, choose a representative $\bar g \in \Gamma_{g}$. 

\STATE For each $v\in \Gamma_{g}$, choose $c_{v,j}\in\mathbb Z$ such that
\[
v =\bar g+\sum\nolimits_{j=1}^d c_{v,j}e_j.
\]

\STATE Let $\bar c_{g,j} := \min\{ c_{v,j} \mid v\in \Gamma_{g}\}$. 

\RETURN
$$
\biggl \{ h_{g}:=\bar g+\sum\nolimits_{j=1}^d \bar c_{g,j}e_j,\ \ 
I_{g}:=K[A]\{t^{v-h_{g}}\mid v\in \Gamma_{g}\}
\mid g\in G\biggr\}
$$
\end{algorithmic}
\end{algorithm}

For $v\in \Gamma_{g}$ the element $t^{v-h_{g}}$ is in $K[A]$ because
\[
v-h_{g}= \sum\nolimits_{j=1}^{d}\left(  c_{v,j}-\bar c_{g,j}\right)  e_{j}
\]
is an expression with non-negative integer coefficients. Thus, $I_{g}$ is a monomial ideal of $K[A]$ and  $h_{g}\in G(B)$ for each $g\in G$, as required.

\begin{example}
\label{zerlbeispiel}

Consider $B=\langle(2,0,3),(4,0,1),(0,2,3),(1,3,1),(1,2,2)\rangle
\subset\mathbb{N}^{3}$ and the subsemigroup $A=\langle
(2,0,3),(4,0,1),(0,2,3),(1,3,1)\rangle$. We get the decomposition of $B_{A}$
into equivalence classes $B_{A}=\{0,(2,4,4)\}\cup\{(1,2,2),(3,6,6)\}$.
Choosing shifts $h_{1}=(-2,0,-3)$ and $h_{2}=(-1,2,-1)$ in $G(B)$ we have
\begin{align*}
K[B] &  \cong K[A]\{t^{(2,0,3)},t^{(4,4,7)}\}(-h_{1})\oplus K[A]\{t^{(2,0,3)}%
,t^{(4,4,7)}\}(-h_{2})\\
&  \cong\langle x_{0},x_{1}x_{2}^{2}\rangle(-h_{1})\oplus\langle x_{0},x_{1}x_{2}^{2}\rangle(-h_{2}),
\end{align*}
where $K[A]\cong K[x_{0},x_{1},x_{2},x_{3}]/\langle x_{1}^{2}x_{2}^{3}-x_{0}^{3}x_{3}^{2} \rangle$.
\end{example}

\begin{example}\label{zerlbeispiel2}
Using our implementation of Algorithm $1$ in the \textsc{Macaulay2} package
\textsc{MonomialAlgebras} we compute the decomposition of $\mathbb Q[B]$ over $\mathbb Q[A]$
in case of Example~\ref{zerlbeispiel}:

\texttt{i1: loadPackage "MonomialAlgebras";}

\texttt{i2: A = \{\{2,0,3\},\{4,0,1\},\{0,2,3\},\{1,3,1\}\};}

\texttt{i3: B = \{\{2,0,3\},\{4,0,1\},\{0,2,3\},\{1,3,1\},\{1,2,2\}\};}

\texttt{i4: S = QQ[x\_0 .. x\_4, Degrees=%
$>$%
B];}

\texttt{i5: P = QQ[x\_0 .. x\_3, Degrees=%
$>$%
A];}

\texttt{i6: f = map(S,P);}

\texttt{i7: dc = decomposeMonomialAlgebra f}

\texttt{
\begin{tabular}
[c]{llll}%
\hspace{-0.15in}\texttt{o7: HashTable\{} & \{0,0,0\} & \texttt{=%
$>$
\{ }ideal ( x$_{0},$ x$_{1}$x$_{2}^{2}$ ), & \{-2,0,-3\} \}\\
& \{5,0,0\} & \texttt{=%
$>$
\{ }ideal ( x$_{0},$ x$_{1}$x$_{2}^{2}$ ), & \{-1,2,-1\} \}\}
\end{tabular}
}

\texttt{i8: ring first first values dc}

\texttt{o8: $\frac{\mbox{\texttt{P}}}{\mbox{\texttt{x}}_{1}^{2}\mbox{\texttt{x}}_{2}^{3}-\mbox{\texttt{x}}_{0}^{3}\mbox{\texttt{x}}_{3}^{2}}$}

The keys of the hash table represent the elements of $G$.
\end{example}

\section{Ring-theoretic properties}\label{sec simplicial case}

Recall that a positive affine semigroup $B$ is simplicial if it spans a simplicial cone, or equivalently, there are linearly independent elements $e_1,\ldots,e_d\in B$ with $C(B)=C(\{e_1,\ldots,e_d\})$. Many ring-theoretic properties of semigroup algebras can be determined from the 
combinatorics of the semigroup; see \cite{GSRBB, MHCM, HRLC, PHDPL, RSHFGA}. Here we give characterizations in terms of the decomposition of Theorem~\ref{decomposition}.

\begin{proposition}\label{char}

Let $K$ be a field, $B\subseteq \mathbb N^m$ a simplicial semigroup, and let $A$ be the submonoid of $B$ which is generated by linearly independent elements $e_1,\ldots,e_d$ of $B$ with \mbox{$C(A)=C(B)$}. Let $B_A$ be as above, and $K[B]\cong\bigoplus_{g\in G} I_{g}(-h_g)$ be the output of Algorithm~\ref{algo1} with respect to $A\subseteq B$ using
minimal generators of $A$. We have:

\begin{enumerate}

\item The depth of $K[B]$ is the minimum of the depths of the ideals $I_g$.

\item $K[B]$ is Cohen-Macaulay if and only if every ideal $I_g$ is equal to $K[A]$.

\item $K[B]$ is Gorenstein if and only if $K[B]$ is Cohen-Macaulay and the set of shifts $\{h_g\}_{g\in G}$ has exactly one maximal element with respect to $\leq$ given by $x\leq y$ if there is an element $z\in B$ such that $x+z=y$.

\item $K[B]$ is Buchsbaum if and only if each  ideal $I_g$ is either equal to $K[A]$, or to
 the homogeneous maximal ideal of $K[A]$ and $h_g+b\in B$ for all $b\in \hilb(B)$. 

\item $K[B]$ is normal if and only if for every element $x$ in $B_{A}$ there exist $\lambda_{1},\ldots,\lambda_{d}\in\mathbb{Q}$ with $0\leq\lambda_{i}<1$ for all $i$ such that $x=\sum_{i=1}^{d}\lambda_{i}e_{i}$.

\item $K[B]$ is seminormal if and only if for every element $x$ in $B_{A}$ there exist $\lambda_{1},\ldots,\lambda_{d}\in\mathbb{Q}$ with $0\leq\lambda_{i}\leq1$ for all $i$ such that $x=\sum_{i=1}^{d}\lambda_{i}e_{i}$.

\end{enumerate}

\end{proposition}
\begin{proof}

For every $x\in G(B)$ there are uniquely determined elements $\lambda_{1}^{x},\ldots,\lambda_{d}^{x}
\in\mathbb{Q}$ such that $x=\sum_{j=1}^{d}\lambda_{j}^{x}e_{j}$. Then by construction
\[
h_{g}=\sum\nolimits_{j=1}^{d}\min\left\{ \lambda^{v}_{j}\mid v\in\Gamma_{g}\right\} e_{j}.
\]
Assertion (1) and (2) follow immediately; (2) was already mentioned in \cite[Theorem\,6.4]{RSHFGA}. Assertion (3) can be found in \cite[Corollary\,6.5]{RSHFGA}.\smallskip\newline
(4) Let $I_g$ be a proper ideal, equivalently, $\#\Gamma_g\geq2$. The ideal $I_g$ is equal to the homogeneous maximal ideal of $K[A]$ and $h_g+b\in B$ for all $b\in \hilb(B)$ if and only if $\Gamma_g=\{m+e_{1},\ldots,m+e_{d}\}$ for some $m$ with $m+b\in B$ for all $b\in \hilb(B)$. Now the assertion follows from \cite[Theorem\,9]{GSRBB}.\smallskip\newline
(5) We set $D_{A}=\{x\in G\left(  B\right)  \mid x=\sum_{i=1}^{d}\lambda_{i}%
e_{i},\lambda_{i}\in\mathbb{Q}$ and $0\leq\lambda_{i}<1~\forall i\}$. The ring $K[B]$ is normal if and only if $B=C(B)\cap G(B)$ by \cite[Proposition\,1]{MHCM}. We need to show that $C(B)\cap G(B)\subseteq B$ if and only if $B_A\subseteq D_A$. We have $B_{A}\subseteq D_{A}$ if and only if $D_{A}\subseteq B_{A}$, since $B_{A}$ has $\#G=\#D_{A}$ equivalence classes and by definition of $B_{A}$. Note that $D_{A}\subseteq C(B)\cap G(B)$ and $D_{A}\cap B\subseteq B_{A}$. The assertion follows from the fact that every element $x\in C(B)\cap G(B)$ can be written as $x=x^{\prime}+\sum_{i=1}^{d}n_{i}e_{i}$ for some $x^{\prime}\in D_{A}$ and $n_{i}\in\mathbb{N}$.\smallskip\newline
(6) We set $\bar D_{A}:=\{x\in B \mid x =\sum_{i=1}^{d} \lambda_{i}e_{i}, \lambda_{i}\in\mathbb{Q }\mbox{ and }0\leq\lambda_{i}\leq1~\forall i\}$. By \cite[Proposition\,5.32]{HRLC} and \cite[Theorem\,4.1.1]{PHDPL} $K[B]$ is seminormal if and only if $B_A\subseteq\bar D_A$, provided that ${e_{1},\ldots,e_{d}}\in\hilb(B)$. Otherwise there is a $k\in\{1,\ldots,d\}$ with $e_{k}=e_{k}^{\prime}+e_{k}^{\prime\prime}$ and $e_{k}^{\prime},e_{k}^{\prime\prime}\in B\setminus\{0\}$. We set $A^{\prime}=\langle e_{1},\ldots,e_{k}^{\prime},\ldots,e_{d}\rangle$ and $A^{\prime\prime}=\langle e_{1},\ldots,e_{k}^{\prime\prime},\ldots,e_{d}\rangle$. Clearly $C(A)=C(A^{\prime})=C(A^{\prime\prime})$. We need to show that $B_A\subseteq \bar D_A$ if and only if $B_{A'}\subseteq \bar D_{A'}$.  Let $x\in B_{A}\setminus\bar D_{A}$. If $x-e_{k}^{\prime}\notin B$, then $x\in B_{A^{\prime}}\setminus\bar D_{A^{\prime}}$. If $x-e_{k}^{\prime}\in B$, then $x-e_{k}^{\prime}\in B_{A^{\prime\prime}}\setminus\bar D_{A^{\prime\prime}}$. Let $x\in B_{A'}\setminus \bar D_{A'}$, say $x=\sum _{j\not=k}\lambda_je_j+\lambda_ke_k'$ and $\lambda_j>1$ for some $j$. If $j\not=k$, then $x\in B_A\setminus \bar D_A$. Let $j=k$; consider the element $y=x+e_{k}''-\sum_{j\not=k}n_je_j\in B$ for some $n_j\in \mathbb N$ such that $\sum_{j\not=k} n_j$ is maximal. It follows that $y\in B_A\setminus\bar D_A$ and we are done.
\end{proof}

Note that normality of positive affine semigroup rings can also be tested using the implementation of normalization in the program \textsc{Normaliz} \cite{normaliz}. We remark that from Proposition~\ref{char} it follows that every simplicial affine semigroup ring $K[B]$ which is seminormal and Buchsbaum is also Cohen-Macaulay. This holds more generally for arbitrary positive affine semigroups by \cite[Proposition\,4.15]{BLRSN}.

\begin{example}[Smooth Rational Monomial Curves in $\mathbb{P}^{3}$]

Consider the simplicial semigroup $B=\langle(\alpha,0), (\alpha-1,1), (1,\alpha-1), (0,\alpha)\rangle\subseteq \mathbb N^2$ and set $A=\langle(\alpha,0),(0,\alpha)\rangle$, say $K[A]=K[x,y]$. Note that we have $\alpha$ equivalence classes. We get
\[
K[B]\cong K[x,y]^{3}\oplus\langle x^{\alpha-3},y\rangle\oplus\langle
x^{\alpha-4},y^{2}\rangle\oplus\ldots\oplus\langle x,y^{\alpha-3}\rangle,
\]
as $K[x,y]$-modules, where the shifts are omitted. In the decomposition each ideal of the form $\langle x^{i},y^{j}\rangle$, $1\leq i,j\leq\alpha-3$ with $i+j=\alpha-2$ appears exactly once. Hence $K[B]$ is not Buchsbaum for $\alpha>4$, since $\langle x^{\alpha-3},y\rangle$ is a direct summand. In case $\alpha=4$ there is only one proper ideal $I_{4}=\langle x,y\rangle$ and $h_{4}=(2,2)$; in fact $(2,2)+\hilb(B)\subseteq B$ and therefore $K[B]$ is Buchsbaum. It follows immediately that $K[B]$ is Cohen-Macaulay for $\alpha\leq3$, Gorenstein for $\alpha\leq2$, seminormal for $\alpha\leq3$, and normal for $\alpha\leq3$. Note that we could also decompose $K[B]$ over the subring $K[A]$, where $A=\langle(2\alpha,0),(0,2\alpha)\rangle=K[x',y']$, for $\alpha=4$ we would get
\[
K[B]\cong K[x',y']^{15} \oplus\langle x',y'\rangle
\]
and the corresponding shift of $\langle x',y'\rangle$ is again $(2,2)$.
\end{example}

\begin{example}
Let $B=\langle(1,0,0),(0,1,0),(0,0,2),(1,0,1),(0,1,1)\rangle\subset\mathbb N^3$, moreover, let \mbox{$A=\langle(1,0,0),(0,1,0),(0,0,2)\rangle$}, say $K[A]=K[x,y,z]$. This example was given in \cite[Example\,6.0.2]{PHDPL} to study the relation between seminormality and the Buchsbaum property. We have
\[
K[B]\cong K[A]\oplus\langle x,y\rangle(-(0,0,1)),
\]
as $\mathbb{Z}^{3}$-graded $K[A]$-modules. Hence $K[B]$ is not Buchsbaum,
since $\langle x,y\rangle$ is not maximal; moreover, $K[B]$ is seminormal, but not normal. 

\end{example}

\begin{example}
\label{ex gorenstein}
Consider the semigroup $B=\langle(1,0,0),(0,2,0),(0,0,2),(1,0,1),(0,1,1)\rangle\subset\mathbb N^3$, and set $A=\langle(1,0,0),(0,2,0),(0,0,2)\rangle$. We get
\[
K[B]\cong K[A]\oplus K[A](-(1,0,1))\oplus K[A](-(0,1,1))\oplus K[A](-(1,1,2)).
\]
Hence $K[B]$ is Gorenstein, since $(1,0,1)+(0,1,1)=(1,1,2)$. Moreover, $K[B]$ is not normal, since $(1,0,1)=(1,0,0)+\frac{1}{2}(0,0,2)$, but seminormal.

\end{example}

\begin{example}
We illustrate our implementation of the characterizations given in Proposition~\ref{char} at Example \ref{ex gorenstein}:

\texttt{i1: B = \{\{1,0,0\},\{0,2,0\},\{0,0,2\},\{1,0,1\},\{0,1,1\}\};}

\texttt{i2: isGorensteinMA B}

\texttt{o2: true}

\texttt{i3: isNormalMA B}

\texttt{o3: false}

\texttt{i4: isSeminormalMA B}

\texttt{o4: true}

\noindent Note that there are also commands \texttt{isCohenMacaulayMA} and \texttt{isBuchsbaumMA} available testing the Cohen-Macaulay and the Buchsbaum property, respectively.

\end{example}

\section{Regularity\label{sec simplicial homogeneous case}}

Let $K$ be a field and let $R=K[x_{1},\ldots,x_{n}]$ be a standard graded
polynomial ring, that is, \mbox{$\deg x_i=1$} for all $i=1,\ldots,n$. Let $R_{+}$ be the homogeneous maximal ideal of $R$, and let $M$ be a finitely
generated graded $R$-module. We define the \emph{Castelnuovo-Mumford regularity}
$\reg M$ of $M$ by
\[
\reg M:=\max\left\{  a(H_{R_{+}}^{i}(M))+i\mid i\geq0\right\}  ,
\]
where $a(H_{R_{+}}^{i}(M)):=\max\left\{  n\mid\lbrack H_{R_{+}}^{i}%
(M)]_{n}\not =0\right\}  $ and $a(0)=-\infty$; $H_{R_{+}}^{i}(M)$ denotes the
$i$-th local cohomology module of $M$ with respect to $R_{+}$. Note that
$\reg M$ can also be computed in terms of the shifts in a minimal
graded free resolution of $M$. An important application of the regularity is
that it bounds the degrees in certain minimal Gr\"{o}bner bases
by \cite{BSCDM}. Thus, it is of interest to compute or bound the regularity of
a homogeneous ideal. The following conjecture (Eisenbud-Goto) was made in
\cite{EG}: If $K$ is algebraically closed and $I$ is a homogeneous prime ideal
of $R$ then for $S=R/I$%
\[
\reg S\leq\deg S-\codim S.
\]
Here $\deg S$ denotes the degree of $S$ and $\codim %
S:=\dim_{K} S _{1}-\dim S$ the codimension. The conjecture has been proved
for  dimension $2$ by Gruson, Lazarsfeld, and Peskine
\cite{GLP}; for the Buchsbaum case by St\"{u}ckrad and Vogel \cite{EGBB} (see also
\cite{EGCM}); for $\deg S\leq\codim S+2$ by Hoa, St\"{u}ckrad,
and Vogel \cite{HSVC2}; and in characteristic zero for smooth surfaces and 
certain smooth threefolds by Lazarsfeld \cite{LSMSD3} and Ran \cite{RLDG}. There is also a stronger version in which $S$ is
only required to be reduced and connected in codimension $1$; this version has
been proved in dimension 2 by Giaimo in \cite{DGEGCC}. For homogeneous
semigroup rings of codimension $2$ the conjecture was proved by Peeva
and Sturmfels \cite{codim2}. Even in the simplicial setting the conjecture is largely open,
though it was proved
for the isolated singularity case
 by Herzog and Hibi \cite{CMHH}; for the
seminormal case by \cite{MNSNAX}; and for a few other cases by
\cite{HSCM, MNEGMC}.\newline

We now focus on computing the regularity of a homogeneous semigroup ring $K[B]$. Note that a positive affine semigroup $B$ is homogeneous if and only if there is
a group homomorphism $\deg : G(B)\rightarrow\mathbb{Z}$ with
$\deg b=1$ for all $b\in \hilb(B)$. We always consider the $R$-module
structure on $K[B]$ given by the homogeneous surjective $K$-algebra
homomorphism $R\twoheadrightarrow K[B], x_{i}\mapsto t^{b_{i}}$, where $\hilb(B)=\{b_1,\ldots,b_n\}$. Generalizing the results from \cite{HSCM}, the regularity can be computed in terms of the decomposition of Theorem~\ref{decomposition} as follows: 

\begin{proposition}
\label{regcomp}

Let $K$ be an arbitrary field and let $B\subseteq\mathbb{N}^{m}$ be a
homogeneous  semigroup. Fix a group homomorphism $\deg:
G(B)\rightarrow\mathbb{Z}$ with $\deg b=1$ for all $b\in
\hilb(B)$. Moreover, let $A$ be a submonoid of $B$ with $\hilb(A)=\{e_{1},\ldots,e_{d}\}$, $\deg e_{i}=1$ for all $i$, and
$C(A)=C(B)$. Let $K[B]\cong\bigoplus_{g\in G} I_{g} (-h_{g})$ be the output of
Algorithm~\ref{algo1} with respect to $A\subseteq B$. Then

\begin{enumerate}
\item $\reg K[B] = \max\left\{ \reg I_{g} + \deg h_{g}\mid g\in G\right\} $; where $\reg I_{g}$ denotes the
regularity of the ideal $I_{g}\subseteq K[A]$ with respect to the canonical $K[x_{1},\ldots,x_{d}]$-module structure.

\item $\deg K[B] = \#G \cdot\deg K[A]$.
\end{enumerate}
\end{proposition}

\begin{proof}
(1) Consider the $T=K[x_{1},\ldots,x_{d}]$-module structure on $K[B]$ which is
given by $T\twoheadrightarrow K[A]\subseteq K[B]$,  $x_{i}\mapsto t^{e_{i}}$. Since $C(A)=C(B)$ we get by \cite[Theorem\,13.1.6]{BSLC}
\[
H^{i}_{K[B]_{+}}(K[B])\cong H^{i}_{T_{+}}(K[B]),
\]
as $\mathbb{Z}$-graded $T$-modules (where $K[B]_{+}$ is the homogeneous maximal ideal of $K[B]$). By the same theorem we obtain $H^{i}_{K[B]_{+}}(K[B])\cong H^{i}_{R_{+}}(K[B])$. Then the assertion follows from $K[B]\cong\bigoplus_{g\in G} I_{g}(-\deg h_{g})$ as $\mathbb{Z}$-graded $T$-modules.

(2) Follows from $\deg I_{g}=\deg K[A]$ for all $g\in G$.
\end{proof}

Using Proposition~\ref{regcomp} we obtain Algorithm~\ref{algo2}.

\begin{algorithm}[H]
\caption{The regularity algorithm}
\label{algo2}
\begin{algorithmic}[1]
\REQUIRE The Hilbert basis $\hilb(B)$ of a homogeneous  semigroup $B\subseteq \mathbb N^m$ and a field $K$.
\ENSURE The Castelnuovo-Mumford regularity $\reg K[B]$.
\STATE Choose a minimal subset $\{e_1,\ldots,e_d\}$ of $\hilb(B)$ with $C(\{e_1,\ldots,e_d\})=C(B)$, and set $A=\erz{e_1,\ldots,e_d}$.
\STATE Compute the decomposition $K[B]\cong\bigoplus_{g\in G} I_g(-h_g)$ over $K[A]$ by Algorithm~\ref{algo1}.
\STATE Compute a hyperplane $H=\{(t_1,\ldots,t_m)\in\mathbb R^m \mid \sum_{j=1}^m a_jt_j =c\}$ with $c\not=0$ such that $\hilb(B)\subseteq H$. Define $\deg :\mathbb R^m\rightarrow\mathbb R$ by $\deg (t_1,\ldots,t_m) = (\sum_{j=1}^m a_jt_j)/c$.
\RETURN $\reg K[B] = \max\maxk{\reg I_g + \deg h_g\mid g\in G}$.
\end{algorithmic}
\end{algorithm}

By Algorithm~\ref{algo2} the computation of $\reg K[B]$ reduces to
computing minimal graded free resolutions of the monomial ideals $I_{g}$ in
$K[A]$ as $K[x_{1},\ldots,x_{d}]$-modules.

\begin{example}
We apply Algorithm \ref{algo2} using the decomposition computed in
Example~\ref{zerlbeispiel2}. A resolution of $I=\left\langle x_{0},x_{1}%
x_{2}^{2}\right\rangle $ as a $T=\mathbb Q[x_{0},x_{1},x_{2},x_{3}]
$-module is%
\[
0\longrightarrow T\left( -4\right) \oplus T\left( -5\right) \overset
{d}{\longrightarrow}T\left( -1\right) \oplus T\left( -3\right)
\longrightarrow I\longrightarrow0
\]
with%
\[
d=\left(
\begin{array}
[c]{cc}%
x_{1}x_{2}^{2} & x_{0}^{2}x_{3}^{2}\\
-x_{0} & -x_{1}x_{2}%
\end{array}
\right) ,
\]
hence $\reg I=4$. The group homomorphism is given by $\deg b=(b_1+b_2+b_3)/5$ and therefore $\reg \mathbb Q[B]=\max\left\{4-1,4-0\right\} =4$.
\end{example}

With respect to timings, we first focus on dimension $3$ comparing our
implementation of Algorithm~\ref{algo2} in the \textsc{Macaulay2} package
\textsc{MonomialAlgebras} (marked in the tables by MA) with other methods.
Here we consider the computation of the regularity via a minimal graded free
resolution both in \textsc{Macaulay2} (M2) and \textsc{Singular} \cite{DGPS}
(S). Furthermore, we compare with the algorithm of Bermejo and Gimenez
\cite{BGSCMR}. This method does not require the computation of a free
resolution, and is implemented in the \textsc{Singular} package
\textsc{mregular.lib} \cite{GLP2} (BG-S) and the \textsc{Macaulay2} package
\textsc{Regularity} \cite{SeSt} (BG-M2). For comparability we obtain the
toric ideal $I_{B}$ always through the program \textsc{4ti2} \cite{4ti2},
which can be called optionally in our implementation (using
\cite{4ti2interface}).
We give the average computation times over $n$ examples generated by the
function \texttt{randomSemigroup($\alpha$,d,c,num=%
$>$%
n,setSeed=%
$>$%
true)}. Starting with the standard random seed, this function generates $n$
random semigroups $B\subseteq\mathbb{N}^{d}$ such that

\begin{itemize}
\item $\dim K[B] = d$.

\item $\codim K[B]=c$; that is the number of generators of $B$ is $d+c$.

\item Each generator of $B$ has coordinate sum equal to $\alpha$.
\end{itemize}

All timings are in seconds on a single $2.7$ GHz core and $4$ GB of RAM. In
the cases marked by a star at least one of the computations ran out of memory
or did not finish within $1200$ seconds. Note that the computation of
$\reg I_{g}$ in step $4$ of Algorithm~\ref{algo2} could easily be
parallelized. This is not available in our \textsc{Macaulay2} implementation
so far.

The next table shows the comparison for $K=\mathbb{Q}$, $d=3$, $\alpha=5$, and
$n=15$ examples.%
\[%
\begin{tabular}
[c]{lccccccccc}%
$c$ & \multicolumn{1}{|c}{$1$} & $2$ & $3$ & $4$ & $5$ & $6$ & $7$ & $8$ &
$9$\\\hline
MA & \multicolumn{1}{|c}{$.073$} & $.089$ & $.095$ & $.10$ & $.13$ & $.14$ &
$.14$ & $.19$ & $.16$\\
M2 & \multicolumn{1}{|c}{$.0084$} & $.0089$ & $.011$ & $.017$ & $.043$ & $.10$
& $.45$ & $2.8$ & $21$\\
S & \multicolumn{1}{|c}{$.0099$} & $.0089$ & $.011$ & $.013$ & $.020$ & $.046$
& $.18$ & $1.1$ & $6.8$\\
BG-S & \multicolumn{1}{|c}{$.016$} & $.030$ & $.19$ & $1.2$ & $15$ & $24$ &
$59$ & $44$ & $77$\\
BG-M2 & \multicolumn{1}{|c}{$.036$} & $.053$ & $.47$ & $1.8$ & $9.0$ & $19$ &
$34$ & $39$ & $43$\\
&  &  &  &  &  &  &  &  & \\
$c$ & \multicolumn{1}{|c}{$10$} & $11$ & $12$ & $13$ & $14$ & $15$ & $16$ &
$17$ & $18$\\\hline
MA & \multicolumn{1}{|c}{$.21$} & $.26$ & $.22$ & $.26$ & $.29$ & $.30$ &
$.31$ & $.36$ & $.47$\\
M2 & \multicolumn{1}{|c}{$180$} & $\ast$ & $\ast$ & $\ast$ & $\ast$ & $\ast$ &
$\ast$ & $\ast$ & $\ast$\\
S & \multicolumn{1}{|c}{$30$} & $\ast$ & $\ast$ & $\ast$ & $\ast$ & $\ast$ &
$\ast$ & $\ast$ & $\ast$\\
BG-S & \multicolumn{1}{|c}{$170$} & $520$ & $\ast$ & $\ast$ & $\ast$ & $\ast$
& $360$ & $460$ & $350$\\
BG-M2 & \multicolumn{1}{|c}{$85$} & $150$ & $140$ & $250$ & $310$ & $290$ &
$300$ & $410$ & $320$%
\end{tabular}
\ \
\]

For small codimension $c$ the decomposition approach has slightly higher
overhead than the traditional algorithms. For larger codimensions, however,
both the resolution approach in \textsc{Macaulay2} and \textsc{Singular} and
the Bermejo-Gimenez implementation in \textsc{Singular} fail. The average
computation times of the \textsc{Regularity} package increase significantly,
whereas those for Algorithm~\ref{algo2} stay under one second. The traditional
approaches become more competitive when considering the same setup over the
finite field $K=\mathbb{Z}/101$, but are still much slower than
Algorithm~\ref{algo2}:%

\[%
\begin{tabular}
[c]{lccccccccc}%
$c$ & \multicolumn{1}{|c}{$1$} & $2$ & $3$ & $4$ & $5$ & $6$ & $7$ & $8$ &
$9$\\\hline
MA & \multicolumn{1}{|c}{$.072$} & $.088$ & $.093$ & $.10$ & $.12$ & $.13$ &
$.13$ & $.19$ & $.16$\\
M2 & \multicolumn{1}{|c}{$.0075$} & $.0095$ & $.010$ & $.013$ & $.020$ &
$.032$ & $.090$ & $.40$ & $2.8$\\
S & \multicolumn{1}{|c}{$.0067$} & $.010$ & $.011$ & $.015$ & $.023$ & $.041$
& $.16$ & $.99$ & $6.3$\\
BG-S & \multicolumn{1}{|c}{$.017$} & $.020$ & $.031$ & $.052$ & $.094$ & $.12$
& $.18$ & $.34$ & $.42$\\
BG-M2 & \multicolumn{1}{|c}{$.030$} & $.037$ & $.064$ & $.14$ & $.34$ & $.48$
& $.80$ & $1.5$ & $2.0$\\
&  &  &  &  &  &  &  &  & \\
$c$ & \multicolumn{1}{|c}{$10$} & $11$ & $12$ & $13$ & $14$ & $15$ & $16$ &
$17$ & $18$\\\hline
MA & \multicolumn{1}{|c}{$.21$} & $.25$ & $.22$ & $.25$ & $.29$ & $.29$ &
$.31$ & $.35$ & $.39$\\
M2 & \multicolumn{1}{|c}{$26$} & $\ast$ & $\ast$ & $\ast$ & $\ast$ & $\ast$ &
$\ast$ & $\ast$ & $\ast$\\
S & \multicolumn{1}{|c}{$28$} & $250$ & $\ast$ & $\ast$ & $\ast$ & $\ast$ &
$\ast$ & $\ast$ & $\ast$\\
BG-S & \multicolumn{1}{|c}{$.57$} & $.88$ & $.88$ & $1.1$ & $1.4$ & $1.5$ &
$1.7$ & $2.5$ & $2.4$\\
BG-M2 & \multicolumn{1}{|c}{$3.3$} & $4.4$ & $4.4$ & $6.4$ & $7.9$ & $7.8$ &
$9.2$ & $12$ & $13$%
\end{tabular}
\]
\vspace{0.02cm}

Note that over a finite field there may not exist a homogeneous linear
transformation such that the initial ideal is of nested type, see for example
\cite[Remark 4.9]{BGSCMR}. This case is not covered and hence does not
terminate in the implementation of the Bermejo-Gimenez algorithm in the
\textsc{Regularity} package. In the standard configuration the package \textsc{mregular.lib} can handle this
case, but then does not perform well over a finite field in our setup. Hence
we use its alternative option, which takes the same approach as the
\textsc{Regularity} package and applies a random homogeneous linear transformation.

Increasing the dimension to $d=4$ we compare our implementation with the most
competitive one, that is, \textsc{mregular.lib} ($K=\mathbb{Z}/101$,
$\alpha=5$, $n=1)$. Here also the \textsc{Singular} implementation of the
Bermejo-Gimenez algorithm fails:%
\vspace{0.02cm}
\[%
\begin{tabular}
[c]{lccccccccccccc}%
$c$ & \multicolumn{1}{|c}{$4$} & $8$ & $12$ & $16$ & $20$ & $24$ & $28$ & $32$
& $36$ & $40$ & $44$ & $48$ & $52$\\\hline
MA & \multicolumn{1}{|c}{$.13$} & $.31$ & $3.8$ & $13$ & $.69$ & $2.2$ &
$1.7$ & $1.9$ & $1.5$ & $4.4$ & $6.0$ & $8.9$ & $13$\\
BG-S & \multicolumn{1}{|c}{$.61$} & $2.2$ & $46$ & $150$ & $380$ & $840$ &
$940$ & $\ast$ & $\ast$ & $\ast$ & $\ast$ & $\ast$ & $\ast$%
\end{tabular}
\]
\vspace{0.02cm}

To illustrate the performance of Algorithm~\ref{algo2} we present the
computation times ($K=\mathbb{Z}/101$, $n=1$) of our implementation for $d=3$
and various $\alpha$ and $c$:%
\vspace{0.02cm}
\[%
\begin{tabular}
[c]{c|ccccccccccccc}%
$\alpha\backslash c$ & $4$ & $8$ & $12$ & $16$ & $20$ & $24$ & $28$ & $32$ &
$36$ & $40$ & $44$ & $48$ & $52$\\\hline
$3$ & $.083$ &  &  &  &  &  &  &  &  &  &  &  & \\
$4$ & $.073$ & $.10$ & $.24$ &  &  &  &  &  &  &  &  &  & \\
$5$ & $.11$ & $.13$ & $.15$ & $.22$ &  &  &  &  &  &  &  &  & \\
$6$ & $.11$ & $.31$ & $.21$ & $.22$ & $.27$ & $.75$ &  &  &  &  &  &  & \\
$7$ & $.10$ & $.16$ & $.18$ & $.24$ & $.29$ & $.86$ & $1.0$ & $1.4$ &  &  &  &
& \\
$8$ & $.11$ & $.22$ & $.26$ & $.31$ & $.35$ & $.54$ & $.67$ & $.85$ & $1.2$ &
$3.6$ &  &  & \\
$9$ & $.13$ & $.25$ & $.31$ & $.38$ & $.56$ & $.64$ & $.77$ & $.98$ & $1.4$ &
$3.8$ & $5.7$ & $8.6$ & $13$%
\end{tabular}
\
\]
\vspace{0.02cm}
The following table is based on a similar setup for $d=4$:%
\vspace{0.02cm}
\[%
\begin{tabular}
[c]{c|cccccccccc}%
$\alpha\backslash c$ & $8$ & $16$ & $24$ & $32$ & $40$ & $48$ & $56$ & $64$ &
$72$ & $80$\\\hline
$3$ & $.18$ & $.51$ &  &  &  &  &  &  &  & \\
$4$ & $.26$ & $.32$ & $.54$ &  &  &  &  &  &  & \\
$5$ & $.31$ & $13$ & $2.2$ & $1.9$ & $4.4$ & $8.9$ &  &  &  & \\
$6$ & $9.6$ & $120$ & $\ast$ & $\ast$ & $3.4$ & $7.8$ & $15$ & $36$ & $66$ &
$120$%
\end{tabular}
\]
\vspace{0.02cm}

Obtaining the regularity via Algorithm~\ref{algo2} involves two main
computations - decomposing $K[B]$ into a direct sum of monomial ideals
$I_{g}\subseteq K[A]$ via Algorithm~\ref{algo1} and computing a minimal graded
free resolution for each $I_{g}$. The computation time for the first task is
increasing with the codimension. On the other hand the complexity of the
second task grows with the cardinality of $\hilb(A)$, which tends to be small
for big codimension. This explains the good performance of the algorithm for
large codimension observed in the table above. In particular the simplicial
case shows an impressive performance as illustrated by the following table for
simplicial semigroups with $d=5$ and $\alpha=5$ (same setup as before). The
examples are generated by the function \texttt{randomSemigroup} using the
option \texttt{simplicial=>true}.%

\vspace{0.02cm}
\[%
\begin{tabular}
[c]{c|cccccccccccc}%
$c$ & $10$ & $20$ & $30$ & $40$ & $50$ & $60$ & $70$ & $80$ & $90$ & $100$ &
$110$ & $120$\\\hline
MA & $13$ & $13$ & $17$ & $32$ & $69$ & $86$ & $110$ & $170$ & $250$ & $400$ &
$650$ & $1000$%
\end{tabular}
\]
\vspace{0.02cm}

In case of a homogeneous  semigroup ring of dimension $2$ the ideals
$I_{g}$ are monomial ideals in two variables. Hence we can read off
$\reg I_{g}$ by ordering the monomials with respect to the lexicographic order
(see, for example, \cite[Proposition\,4.1]{MNEGMC}). This further improves the
performance of the algorithm.\newline

Due to the good performance of Algorithm~\ref{algo2} we can actually do the
regularity computation for all possible semigroups $B$ in $\mathbb{N}^{d}$
such that the generators have coordinate sum $\alpha$ for some $\alpha$ and
$d$. This confirms the Eisenbud-Goto conjecture for some cases.

\begin{proposition}
\label{testegbis5} The regularity of $\mathbb{Q}[B]$ is bounded by
$\deg\mathbb{Q}[B]-\codim \mathbb{Q}[B]$, provided that the minimal generators
of $B$ in $\mathbb{N}^{d}$ have fixed coordinate sum $\alpha$ for $d=3$ and
$\alpha\leq5$, for $d=4$ and $\alpha\leq3$, as well as for $d=5$ and
$\alpha=2$.
\end{proposition}

\begin{proof}
The list of all generating sets $\hilb(B)$ together with $\reg \mathbb{Q}[B]$,
$\deg\mathbb{Q}[B]$, and $\codim \mathbb{Q}[B]$ can be found under the link
given in \cite{BEN}.
\end{proof}

Figure \ref{projr34} depicts the values of $\deg\mathbb{Q}[B]-\codim\mathbb{Q}%
[B]$ plotted against $\reg\mathbb{Q}[B]$ for all semigroups with $\alpha=3$
and $d=4$. For the same setup Figure \ref{egbar34} shows $\reg\mathbb{Q}[B]$
on top of $\codim\mathbb{Q}[B]$ plotted against $\deg\mathbb{Q}[B]$. The line
corresponds to the projection of the plane $\reg\mathbb{Q}[B]-\deg
\mathbb{Q}[B]+\codim\mathbb{Q}[B]=0$. Figures for the remaining cases can be
found at \cite{BEN}.

\begin{figure}[ptb]
\begin{center}
\includegraphics[trim=0cm 0cm 0cm 0cm,
height=3.6in,
width=3.2in
]{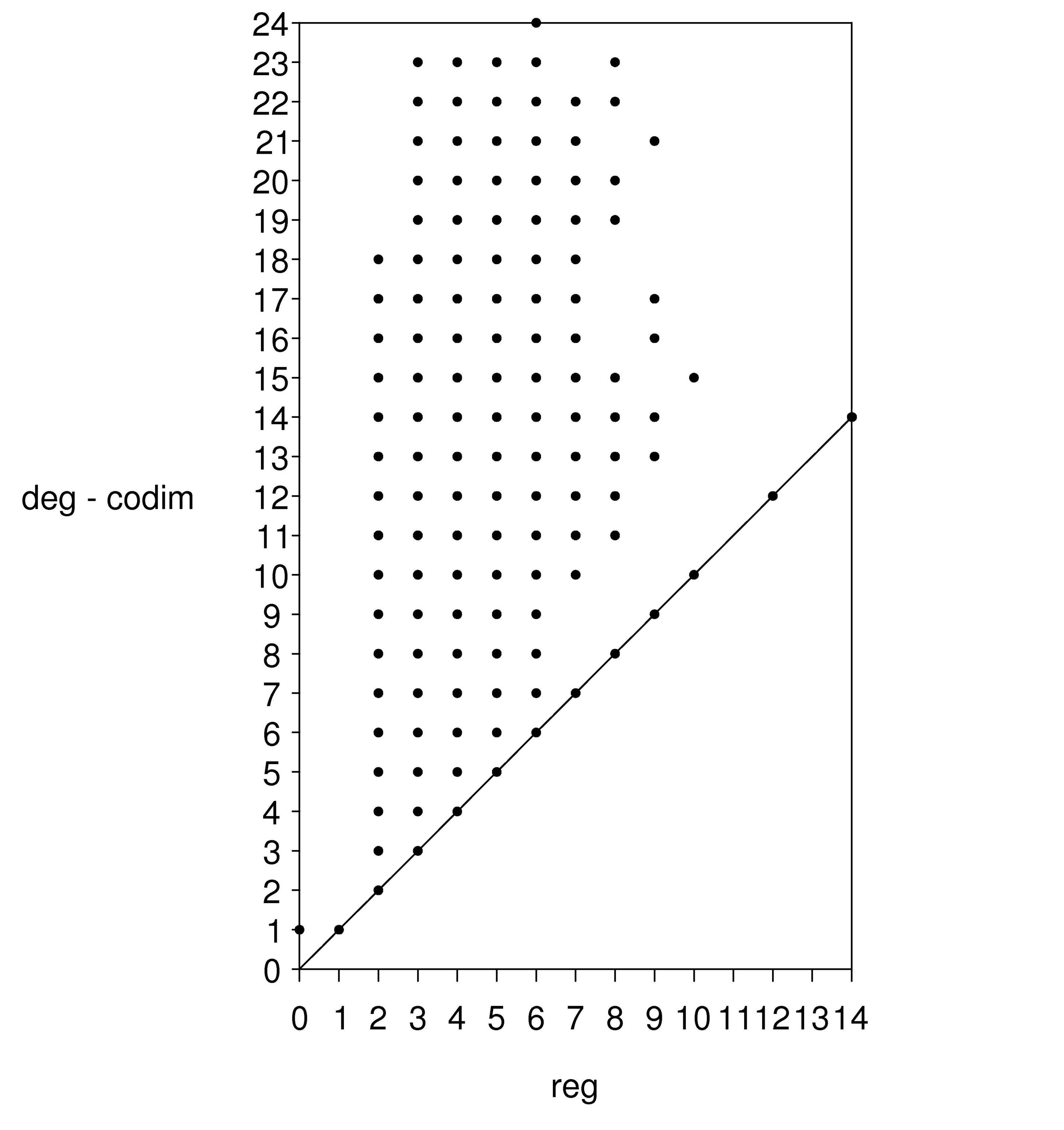}
\end{center}
\caption{$\deg\mathbb{Q}[B] -\codim\mathbb{Q}[B] $ against $\reg\mathbb{Q}[B]
$ for $\alpha=3$ and $d=4$.}%
\label{projr34}%
\end{figure}

\begin{figure}[ptb]
\begin{center}
\includegraphics[trim=0cm 0cm 0cm 0cm,
height=3.5in,
width=4.75in
]{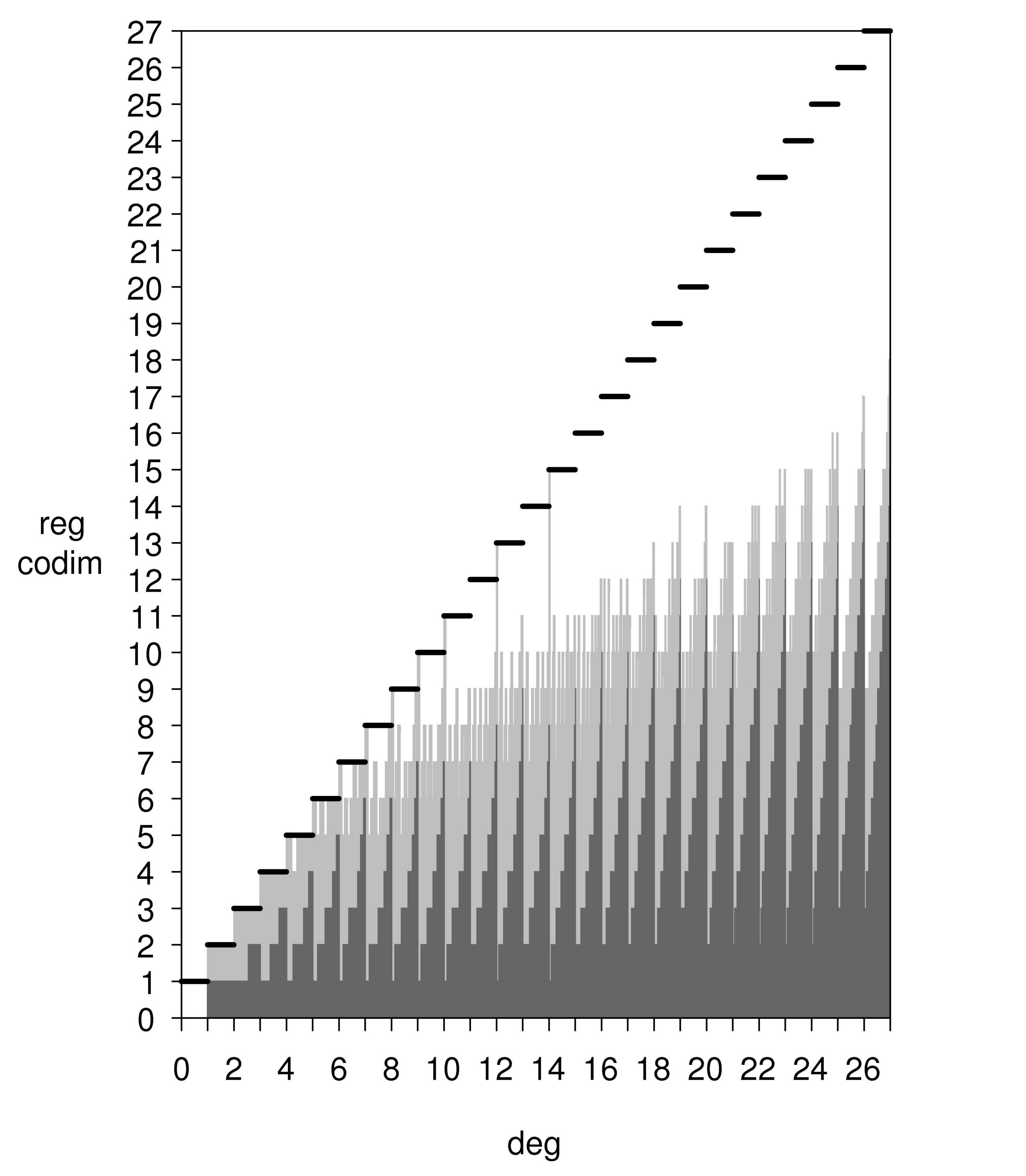}
\end{center}
\caption{$\reg\mathbb{Q}[B] +\codim\mathbb{Q}[B] $ against $\deg\mathbb{Q}[B]
$ for $\alpha=3$ and $d=4$.}%
\label{egbar34}%
\end{figure}

\end{document}